\documentclass[12pt]{amsart}

  \baselineskip 1 mm

\usepackage{amsmath,amsthm,amsfonts,latexsym,amsopn,verbatim,amscd,amssymb}
\usepackage{hyperref}

\theoremstyle{plain}
\newtheorem{Thm}{Theorem}[section]
\newtheorem{Lem}[Thm]{Lemma}
\newtheorem{Prop}[Thm]{Proposition}
\newtheorem{Cor}[Thm]{Corollary}
\theoremstyle{definition}

\newtheorem{Def}[Thm]{Definition}
\newtheorem{example}[Thm]{Example}

\numberwithin{equation}{section}

\newcommand{\bnum}{\begin{enumerate}}
\newcommand{\enum}{\end{enumerate}}

\begin{document}

\title{Weak Nil Clean Ideal}
\author[D. K. Basnet and A. Sharma]{Dhiren Kumar Basnet* and Ajay Sharma}
\address{\noindent D. K. Basnet and A. Sharma\newline Department of Mathematical Sciences,\newline Tezpur University,\newline  Napaam-784028, Sonitpur,\newline Assam, India.}
\email{dbasnet@tezu.ernet.in, ajay123@tezu.ernet.in}

\keywords{Clean ideals, weakly clean ideals, uniquely clean ideal, weakly uniquely clean ideal.\\
\,\,* \emph{Corresponding Author}}

%


%
%
\begin{abstract}
 As a generalization of nil clean ideal, we define weak nil clean ideal of a ring. An ideal $I$ of a ring $R$ is weak nil clean ideal if for any $x\in I$, either $x=e+n$ or $x=-e+n$, where $n$ is a nilpotent element and $e$ is an idempotent element of $R$. Some interesting properties of weak nil clean ideal and its relation with weak nil clean ring have been discussed.
\end{abstract}

\maketitle

%
%

\section{Introduction} \label{S:intro}

In this article, rings are associative with unity. The Jacobson radical, set of units, set of idempotents, set of nilpotent elements and centre of a ring $R$ are denoted by $J(R)$, $U(R)$, $Idem(R)$, $Nil(R)$ and $C(R)$ respectively. Nicholson \cite{nicholson1977lifting} defined an element $x$ of a ring $R$ to be a clean element, if $x=e+u$ for some $e\in Idem(R)$, $u\in U(R)$ and called the ring $R$ as clean ring if all its elements are clean. Diesl \cite{Diesl} defined  a ring $R$ to be nil clean ring if every element of $R$ can be written as a sum of an idempotent and a nilpotent element of $R$. Weakening the condition of clean ring, Ahn and Anderson \cite{ahn2006weakly} defined a ring $R$ to be weakly clean, if every $x\in R$ can be expressed as $x=u+e$ or $x=u-e$, where $u\in U(R)$, $e\in Idem(R)$. Also Basnet and Bhattacharyya \cite{WNCR} defined a ring $R$ to be weak nil clean, if every element $x\in R$ can be written as $x=n+e$ or $x=n-e$, where $n\in Nil(R)$ and $e\in Idem(R)$. H. Chen and M. Chen \cite{chen2003clean} defined an ideal $I$ of a ring $R$ to be clean ideal, if for any $x\in I$, $x=u+e$, for some $u\in U(R)$ and $e\in Idem(R)$. They proved that every ideal having stable range one of a regular ring is clean. Following the idea of clean ideal, Sharma and Basnet \cite{Ajay1} defined an ideal $I$ of a ring $R$ to be nil clean ideal, if for any $x\in I$, $x=n+e$, where $n\in Nil(R)$ and  $e\in Idem(R)$. They proved that for a nil clean expression of an element of a nil clean ideal of a ring $R$, the nilpotent and idempotent elements are actually elements of the ideal. Also they characterized a nil clean ring with its nil clean ideals. As a generalization of clean ideal, Sharma and Basnet \cite{Ajay2} also introduced weakly clean ideal of a ring. An ideal $I$ of a ring $R$ is said to be weakly clean ideal, if for any $x\in I$, $x=u+e$ or $x=u-e$, where $u\in U(R)$ and  $e\in Idem(R)$.

In this article we introduce the notion of weak nil clean ideal as a generalization of nil clean ideal. An ideal $I$ of a ring $R$ is called weak nil clean ideal if for each $a\in I$, either $a=e+n$ or $a=-e+n$, where $e\in Idem(R)$ and $n\in Nil(R)$. Here also we proved that for a weak nil clean expression of an element of a weak nil clean ideal of a ring $R$, the nilpotent and idempotent elements are actually elements of the ideal. Further we characterized a weak nil clean ring with its weak nil clean ideal and nil clean ideal of $R$. Also we discuss some interesting properties of weak nil clean ideals.

\section{Weak nil clean ideal}
\begin{Def}
  An ideal $I$ of a ring $R$ is called weak nil clean ideal of $R$ if for any $x\in I$, there exist $e\in Idem(R)$ and $n\in Nil(R)$ such that $x=e+n$ or $x=-e+n$. Also $I$ is called uniquely weak nil clean ideal of $R$ if for any $x\in I$ there exists a unique $e\in Idem(R)$ such that $x-e\in Nil(R)$ or $x+e\in Nil(R)$.
\end{Def}
Clear from the definition that every ideal of a weak nil clean ring is weak nil clean ideal. But there are non weak nil clean rings which contains some weak nil clean ideals, for example the ring $\mathbb{Z}_{p^n}$, where $n>1$ and $p>5$, a prime number, is not weak nil clean ring but every proper ideal of $\mathbb{Z}_{p^n}$ is weak nil clean ideal. Another such example is given below.
  \begin{example}
    Let $R_1$ be weak nil clean ring and $R_2$ be non weak nil clean ring. Then $R=R_1\oplus R_2$ is not a weak nil clean ring. But clearly $I=R_1\oplus 0$ is weak nil clean ideal of $R$.
  \end{example}
  Observe that every nil clean ideal is weak nil clean ideal but the converse is not true as $\{0,2,4\}$ is weak nil clean ideal of $\mathbb{Z}_6$ but not nil clean ideal of $\mathbb{Z}_6$.
  \begin{Lem}\label{L1}
    Every weak nil clean ideal of a ring $R$ is weakly clean ideal of $R$.
  \end{Lem}
  \begin{proof}
    Let $I$ be a weak nil clean ideal of $R$. For $x\in I$, either $x=e+n$ or $x=-e+n$, where $e\in Idem(R)$ and $n\in Nil(R)$. If $x=e+n$ then $x=(1-e)+(2e-1+n)$ and if $x=-e+n$ then $x=(1-e)+(-1+n)$, where $1-e \in Idem(R)$ and $-1+n$, $2e-1+n\in U(R)$.
  \end{proof}
  The converse of Lemma \ref{L1} is not true as ideal $\{0,3,6,9,12\}$ of $\mathbb{Z}_{15}$ is weakly clean ideal but not weak nil clean ideal of $\mathbb{Z}_{15}$.
  \begin{Prop}\label{PPP1}
  If $I$ is a weak nil clean ideal of a ring $R$ then $I\cap J(R)$ is a nil ideal of $R$.
\end{Prop}
\begin{proof}
  Let $x\in I\cap J(R)$, so either $x=e+n$ or $x=-e+n$, where $e\in Idem(R)$ and $n\in Nil(R)$. If $x=e+n$, then by Proposition $2.4$ \cite{Ajay1}, $x=n$. If $x=-e+n$ then there exists $k\in \mathbb{N}$ such that $n^k=0$. Now $n^k=(x+e)^k=\sum_{t,r\in R}^{finite}txr +e^k=0\Rightarrow e=-\sum_{t,r\in R}^{finite}txr\in J(R)$ as $x\in J(R)$. So $1-e\in Idem(R)\cap U(R)=\{1\}$, hence $1-e=1\Rightarrow e=0\Rightarrow x=n$. Thus the result follows.
\end{proof}
\begin{Cor}
  If $R$ is a weak nil clean ring then $J(R)\subseteq N(R)$. In particular for is a commutative ring $R$, $J(R)=N(R)$.
\end{Cor}
Let $R$ be a ring. An element $a\in R$ is called weakly clean element of type-I if $a=e+u$ and called weakly clean element of type-II if $a=-e+u$, where $e\in Idem(R)$ and $u\in U(R)$.
\begin{Def}
An ideal $I$ of a ring $R$ is said to be strongly weak nil clean ideal if for any $x\in I$, there exist $e\in Idem(R)$ and $n\in Nil(R)$ such that $x=e+n$ or $x=-e+n$ and $en=ne$. Also $I$ is called strongly weakly clean ideal if for any $x\in I$, there exist $e\in Idem(R)$ and $u\in U(R)$ such that $x=e+u$ or $x=-e+u$ and $eu=ue$.
\end{Def}
\begin{Thm}
   Let $I$ be an ideal of $R$ then
   \begin{enumerate}
   \item If $I$ is strongly weak nil clean ideal then it is strongly weakly clean ideal and $a-a^2$ or $a+a^2$ is nilpotent.
   \item If $I$ is strongly weakly clean ideal and $a-a^2$ or $a+a^2$ is nilpotent provided $a$ is of type-I or type-II weakly clean element of $R$ respectively, then $I$ is strongly weak nil clean ideal.
   \end{enumerate}
 \end{Thm}
 \begin{proof}
 \begin{enumerate}
 \item Let $I$ be a strongly weak nil clean ideal and $a\in R$. Then either $a=e+n$ or $a=-e+n$, where $e\in Idem(R)$, $n\in Nil(R)$ and $en=ne$. If $a=e+n$ then $a=(1-e)+(2e-1+n)$ is strongly weakly clean decomposition of $a$ and also $a-a^2=(1-2e-n)n$ is nilpotent. If $a=-e+n$, then $a=-(1-e)+(1-2e+n)$ is strongly weakly clean decomposition of $a$ and also $a+a^2=(1-2e+n)n$ is nilpotent.
 \item Let $a\in I$, so either $a=e+u$ or $a=-e+u$, where $e\in Idem(R)$ and $u\in U(R)$. If $a=e+u$, then $a-a^2$ is nilpotent, which implies $1-2e-u$ is nilpotent and $a=(1-e)+(-1+2e+u)$, a strongly weak nil clean expression of $a$. If $a=-e+u$, then $a+a^2$ is nilpotent, which implies $1-2e+u$ is nilpotent and $a=-(1-e)+(1-2e+u)$, a strongly weak nil clean expression of $a$.
 \end{enumerate}
 \end{proof}
 \begin{Lem}
  Every idempotent in a uniquely weak nil clean ideal is a central idempotent.
\end{Lem}
\begin{proof}
  Let $I$ be a uniquely weak nil clean ideal of a ring $R$ and $e$ be any idempotent in $I$. For any $x\in R$, since $e=(e+ex(1-e))+(-ex(1-e))=e+0$, so $e+ex(1-e)=e\Rightarrow ex=exe$, as $e+ex(1-e)\in Idem(R)$. Similarly we can show that $xe=exe$. Hence $xe=ex$.
\end{proof}
The following theorem shows that, for a weak nil clean expression of an element of a weak nil clean ideal of a ring $R$, the nilpotent and idempotent elements are actually elements of the ideal.
 \begin{Thm}\label{111}
   An ideal $I$ of a ring $R$ is weak nil clean ideal \textit{if and only if} for any $x\in I$, either $x=e+n$ or $x=-e+n$, where $e\in Idem(I)$ and $n\in Nil(I)$.
 \end{Thm}
 \begin{proof}
   Let $I$ be a weak nil clean ideal of $R$ and $x\in I$. There exist $n\in Nil(R)$ and $e\in Idem(R)$ such that either $x=e+n$ or $x=-e+n$. So $n^k=0$, for some $k\in \mathbb{N}$. If $x=e+n$, then $(x-n)^k=(-1)^kn^k+\sum_{i=1}^{s} q_ixp_i$, for some $p_i, q_i\in R$, $(x-n)^k=\sum_{i=1}^{s} q_ixp_i\in I$, so $e^k=e\in I$. Similarly if $x=-e+n$, then also we get $e\in I$. Hence $n\in I$, as required.
 \end{proof}
 The following corollary is immediate.
  \begin{Cor}
   If $R$ is a local ring, then every proper weak nil clean ideal of $R$ is nil ideal. In fact if $R$ has no non trivial idempotents, then every proper weak nil clean ideal of $R$ is nil ideal.
 \end{Cor}
  In the following Theorem, we characterize weak nil clean ring $R$ by weak nil clean ideal and nil clean ideal of $R$.
 \begin{Thm}\label{main}
     $R$ is a weak nil clean ring \textit{if and only if} there exists a central idempotent $e$ in $R$ such that ideals generated by $e$ and $1-e$ are both weak nil clean ideals of $R$ and one of them is nil clean ideal of $R$.
 \end{Thm}
 \begin{proof}
   If $R$ is weak nil clean ring, then $e=1$ works. Conversely, without loss of generality assume that $<e>$ is weak nil clean ideal and $<1-e>$ is nil clean ideal of $R$. For $x\in R$, since $R=<e>+<1-e>$, so $x=a+b$, where $a\in <e>$ and $b\in <1-e>$. There exist $f_1\in Idem(<e>)$ and $n_1\in Nil(<e>)$, such that either $a=f_1+n_1$ or $a=-f_1+n_1$. If $a=f_1+n_1$, then we set $b=f_2+n_2$, where $f_2\in Idem(<1-e>)$ and $n_2\in Nil(<1-e>)$, then $x=(f_1+f_2)+(n_1+n_2)$ is a nil clean expression of $x$ in $R$. Also if $a=-f_1+n_1$ then we set $b=-f_2+n_2$, where $f_2\in Idem(<1-e>)$ and $n_2\in Nil(<1-e>)$, so $x=-(f_1+f_2)+(n_1+n_2)$ is a weak nil clean expression of $x$ in $R$. Hence $R$ is a weak nil clean ring.
 \end{proof}
 A finite orthogonal set of idempotents $e_1, \cdot\cdot\cdot , e_n$ in a ring $R$, is said to be complete set if $e_1+ \cdot\cdot\cdot +e_n=1$. Now we generalize the above result in terms of complete set of idempotents.
  \begin{Thm}
 Ring $R$ is weak nil clean \textit{if and only if} there exists a complete set of central idempotents $e_1, \cdot\cdot\cdot , e_n$ in $R$, such that ideal generated by $e_i$ is weak nil clean ideal of $R$ for all $i$ and at most one $<e_i>$ is not nil clean ideal.
 \end{Thm}
 \begin{proof}
   $(\Rightarrow)$ Taking $e=1$.\\
   $(\Leftarrow)$ Clearly $<e_1>+<e_2>+\cdot\cdot\cdot+<e_n>=R$, so similar to the proof of Theorem \ref{main}, we can show that $R$ is weak nil clean ring.
 \end{proof}
\begin{Prop}
   Let $I$ be an ideal of a ring $R$. Then the following are equivalent:
   \begin{enumerate}
     \item $I$ is weak nil clean ideal of $R$.
     \item There exists a complete set of central idempotents $e_1, \cdot\cdot\cdot , e_n$ such that $e_iI$ is a weak nil clean ideal of $e_iR$, for all $i$ and at most one $e_iI$ is not nil clean ideal of $e_iR$.
   \end{enumerate}
 \end{Prop}
 \begin{proof}
   (1)$ \Rightarrow $ (2) Taking $e=1$.\\
   (2)$\Rightarrow $(1) Let $e_1, \cdot\cdot\cdot , e_n$ be a complete set of idempotents in $R$ such that $e_iI$ is a weak nil clean ideal of $e_iR$, for all $i$ and at most one $e_iI$ is not nil clean ideal of $e_iR$. It is enough to show the result for $n=2$. Clearly $I=e_1I\oplus e_2I$. Without loss of generality assume $e_1I$ is a nil clean ideal of $e_1R$. Let $x\in I$, then $x=a+b$, where $a\in e_1I$ and $b\in e_2I$, so there exist $f_2\in Idem(e_2I)$ and $m_2\in Nil(e_2I)$ such that either $b=f_2+m_2$ or $b=-f_2+m_2$. If $b=f_2+m_2$, then we set $a=f_1+m_1$, where $f_1\in Idem(e_1I)$ and $m_1\in Nil(e_1I)$ and we get $x=(f_1+f_2)+(m_1+m_2)$ is a weak nil clean expression of $x$. If $b=-f_2+m_2$, then we set $a=-f_1+m_1$, where $f_1\in Idem(e_1I)$ and $m_1\in Nil(e_1I)$ and we get $x=-(f_1+f_2)+(m_1+m_2)$ is a weak nil clean expression of $x$.
    \end{proof}

 \begin{Thm}
   Let $R$ be a ring and $I_1$ be an ideal containing the nil ideal $I$. Then $I_1$ is weak nil clean ideal of $R$ \textit{if and only if} $I_1/I$ is weak nil clean ideal of $R/I$.
 \end{Thm}
 \begin{proof}
   If $I_1$ is weak nil clean ideal of $R$, then clearly $I_1/I$ is weak nil clean ideal of $R/I$. Conversely, let $I_1/I$ be weak nil clean ideal of $R/I$ and $x\in I_1$. Then either $\overline{x}=\overline{e}+\overline{n}$ or $\overline{x}=-\overline{e}+\overline{n}$, where $\overline{e}\in Idem(I_1/I)$ and $\overline{n}\in Nil(I_1/I)$. Since idempotents can be lifted modulo nil ideal, so lift $\overline{e}$ to $e\in I_1$. Then $x-e$ or $x+e$ is nilpotent in $I_1$, modulo $I$ and hence $x-e$ or $x+e$ is nilpotent in $I_1$.
 \end{proof}

  \begin{Thm}\label{HOMIMAGE}
   Every homomorphic image of weak nil clean ideal of a ring is also weak nil clean ideal.
 \end{Thm}
 \begin{Thm}\label{Th1}
  Let $\{R_i\}_{i=1}^m$ be a family of rings and $I_i's$ are ideals of $R_i$, then the ideal $I=\prod_{i=1}^{m} I_i$ of $R=\prod_{i=1}^{m} R_i$ is weak nil clean ideal \textit{if and only if} each $I_i$ is weak nil clean ideal of $R_i$ and at most one $I_i$ is not nil clean ideal.
\end{Thm}
\begin{proof}
$(\Rightarrow)$  Let $I$ be weak nil clean ideal of $R$. Then being homomorphic image of $I$ each $I_{\alpha}$ is weak nil clean ideal of $R_{\alpha}$. Suppose $I_{\alpha_1}$ and $I_{\alpha_2}$ are not nil clean ideals, where ${\alpha}_1\neq {\alpha}_2$. Since $I_{\alpha_1}$ is not nil clean ideal, so not all elements $x\in I_{\alpha_1}$ are of the form $x=n-e$, where $n\in Nil(R_{\alpha_1})$ and $e\in Idem(R_{\alpha_1})$. As $I_{\alpha_1}$ is weak nil clean ideal of $R_{\alpha_1}$, so there exists $x_{\alpha_1}\in I_{\alpha_1}$ with $x_{\alpha_1}=n_{\alpha_1}+e_{\alpha_1}$, where $n_{\alpha_1}\in Nil(R_{\alpha_1})$ and $e_{\alpha_1}\in Idem(R_{\alpha_1})$, but $x_{\alpha_1}\neq n-e$, for any $n\in Nil(R_{\alpha_1})$ and $e\in Idem(R_{\alpha_1})$. Similarly there exists $x_{\alpha_2}\in I_{\alpha_2}$ with $x_{\alpha_2}=n_{\alpha_2}-e_{\alpha_2}$, where $n_{\alpha_2}\in Nil(R_{\alpha_2})$ and $e_{\alpha_2}\in Idem(R_{\alpha_2})$, but $x_{\alpha_2}\neq n+e$, for any $n\in Nil(R_{\alpha_2})$ and $e\in Idem(R_{\alpha_2})$. Define $x=(x_\alpha)\in I$ by \begin{align*}
    x_\alpha &=x_{\alpha}\,\, \,\,\,\,\,\,\,if\,\, \alpha \in \{\alpha_1,\alpha_2\} \\
             &=0 \,\,\,\,\,\,\,\,\,\,\,\,\,if \,\,\alpha \notin \{\alpha_1,\alpha_2\}
  \end{align*}
  Then clearly $x\neq n\pm e$, for any $n\in Nil(R)$ and $e\in Idem(R)$. Hence at most one $I_\alpha$ is not nil clean ideal.\par
  $(\Leftarrow)$ If each $I_\alpha$ is nil clean ideal of $R_{\alpha}$ then $I=\prod I_\alpha$ is nil clean ideal of $R$ by Theorem $2.20$ \cite{Ajay1} and hence weak nil clean ideal of $R$. Assume $I_{\alpha_0}$ is weak nil clean ideal but not nil clean ideal of $R_{\alpha_0}$ and that all other $I_{\alpha}$'s are nil clean ideals of $R_{\alpha}$. If $x=(x_\alpha) \in I$, then in $I_{\alpha_0}$, we can write $x_{\alpha_0}=n_{\alpha_0}+e_{\alpha_0}$ or $x_{\alpha_0}=n_{\alpha_0}-e_{\alpha_0}$, where $n_{\alpha_0}\in Nil(R_{\alpha_0})$ and $e_{\alpha_0}\in Idem(R_{\alpha_0})$. If $x_{\alpha_0}=n_{\alpha_0}+e_{\alpha_0}$, then for $\alpha \neq \alpha_0$ we set, $x_{\alpha}=n_{\alpha}+e_{\alpha}$, where $n_{\alpha}\in Nil(I_\alpha)$ and $e_{\alpha}\in Idem(I_\alpha)$. If $x_{\alpha_0}=n_{\alpha_0}-e_{\alpha_0}$, then for $\alpha \neq \alpha_0$ we set, $x_{\alpha}=n_{\alpha}-e_{\alpha}$, where $n_{\alpha}\in Nil(I_\alpha)$ and $e_{\alpha}\in Idem(I_\alpha)$; then $n=(n_\alpha)\in Nil(R)$ and $e=(e_\alpha)\in Idem(R)$, such that $x=n+e$ or $x=n-e$ and hence $I$ is weak nil clean ideal of $R$.
\end{proof}
 If the collection of rings is infinite then Theorem \ref{Th1} is not true as shown by the following example.
 \begin{example}
   Consider the ring $R=\mathbb{Z}_3\times \mathbb{Z}_{2^2}\times \mathbb{Z}_{2^3}\times \cdot\cdot\cdot $, clearly for any $n\in \mathbb{N}$, $\mathbb{Z}_{2^n}$ is weak nil clean ring and hence ideal generated by $2$ in $\mathbb{Z}_{2^n}$ say $<2>_n$ is also weak nil clean ideal of $\mathbb{Z}_{2^n}$. But the ideal $I=<3>\times <2>_2\times <2>_3\times \cdot\cdot\cdot$ is not weak nil clean ideal of $R$ as $(3,2,2,\cdot\cdot\cdot)\in I$ can not be written as a sum of an idempotent and a nilpotent element of $R$.
 \end{example}
 Next we study the relationship between weak nil clean ideal of a given ring $R$ and weak nil clean ideal of upper triangular matrix ring $\mathbb{T}_n(R)$. Here given a matrix $X$, $X_{ij}$ denotes the $(i,j)^{th}$ entry of $X$.
  \begin{Lem}\label{D211}
   For $E, N\in \mathbb{T}_n(R)$ the following hold:
   \begin{enumerate}
     \item If $E^2=E$, then $(E_{ii})^2=E_{ii}$ for $1\leq i\leq n$.
     \item $N$ is nilpotent if and only if $N_{ii}$ is nilpotent for $1\leq i\leq n$.
   \end{enumerate}
 \end{Lem}
\begin{proof}
  See Lemma $2.1.1$ \cite{Diesl1}.
\end{proof}
 \begin{Thm}
 Let $T$ be a $2\times2$ upper triangular matrix ring over $R$. Then an ideal $S=\left(
                                                                 \begin{array}{cc}
                                                                   I & R \\
                                                                   0 & J \\
                                                                 \end{array}
                                                               \right)$ of $T$ is weak nil clean ideal \textit{if and only if} $I$ and $J$ are weak nil clean ideals of $R$ and one of them is nil clean ideal.
\end{Thm}
\begin{proof}
  Without loss of generality, assume that $I$ and $J$ are respectively nil clean and weak nil clean ideals of $R$. Let $x=\left(
                                                                                                                    \begin{array}{cc}
                                                                                                                      a & r \\
                                                                                                                      0 & b \\
                                                                                                                    \end{array}
                                                                                                                  \right)\in \left(
                                                                                                                               \begin{array}{cc}
                                                                                                                                 I & R \\
                                                                                                                                 0 & J \\
                                                                                                                               \end{array}
                                                                                                                             \right)
  $. So either $b=f+n_1$ or $b=-f+n_1$, where $f\in Idem(J)$ and $n_1\in Nil(J)$. If $b=f+n_1$, set $a=e+n$, where $e\in Idem(I)$ and $n\in Nil(I)$. Then $x=\left(
                                                                                                                                   \begin{array}{cc}
                                                                                                                                     e & 0 \\
                                                                                                                                     0 & f \\
                                                                                                                                   \end{array}
                                                                                                                                 \right)+\left(
                                                                                                                                           \begin{array}{cc}
                                                                                                                                             n & r \\
                                                                                                                                             0 & n_1 \\
                                                                                                                                           \end{array}
                                                                                                                                         \right)
  $, where $\left(
                                                                                                                                   \begin{array}{cc}
                                                                                                                                     e & 0 \\
                                                                                                                                     0 & f \\
                                                                                                                                   \end{array}
                                                                                                                                 \right)\in Idem(T)$
  and $\left(
         \begin{array}{cc}
           n & r \\
           0 & n_1 \\
         \end{array}
       \right)\in Nil(T)
  $. If $b=-f+n_1$, set $a=-e+n$, where $e\in Idem(I)$ and $n\in Nil(I)$. Then $x=-\left(
                                                                                                                                   \begin{array}{cc}
                                                                                                                                     e & 0 \\
                                                                                                                                     0 & f \\
                                                                                                                                   \end{array}
                                                                                                                                 \right)+\left(
                                                                                                                                           \begin{array}{cc}
                                                                                                                                             n & r \\
                                                                                                                                             0 & n_1 \\
                                                                                                                                           \end{array}
                                                                                                                                         \right)
  $, where $\left(
                                                                                                                                           \begin{array}{cc}
                                                                                                                                             n & r \\
                                                                                                                                             0 & n_1 \\
                                                                                                                                           \end{array}
                                                                                                                                         \right)\in Nil(T)$
  and $\left(
                                                                                                                                   \begin{array}{cc}
                                                                                                                                     e & 0 \\
                                                                                                                                     0 & f \\
                                                                                                                                   \end{array}
                                                                                                                                 \right)\in Idem(T)$ by
  Lemma \ref{D211}. \\
  For the converse, clearly $I$ and $J$ are weak nil clean ideals of $R$. Suppose both are not nil clean ideals of $R$. As $I$ is not weak nil clean ideal of $R$, so there exists $x\in I$ such that $x=e_1+n_1$, where $e_1\in Idem(I)$ and $n_1\in Nil(I)$ but $x\neq n-e$, for all $n\in Nil(I)$ and $e\in Idem(I)$. Similarly there exists $y\in J$ such that $y=-e_2+n_2$, where $e_2\in Idem(J)$ and $n_2\in Nil(J)$ but $y\neq n+e$, for all $n\in Nil(J)$ and $e\in Idem(J)$. Then it is easy to see that $\left(
                                          \begin{array}{cc}
                                            x & 0 \\
                                            0 & y \\
                                          \end{array}
                                        \right)
  $ is not weak nil clean element of $T$.
\end{proof}
Let $R$ be a commutative ring and $M$ be a $R$-module. Then the idealization of $R$ and $M$ is the ring $R(M)$ with underlying set $R\times M$ under coordinatewise addition and multiplication given by $(r,m)(r',m')=(rr', rm'+r'm)$, for all $r, r'\in R$ and $m, m' \in M$. It is obvious that if $I$ is an ideal of $R$ then for any submodule $N$ of $M$, $I(N)=\{(r,n)\, : \,r\in I \,\, ,\, \,n\in N \}$ is an ideal of $R(M)$. First we mention basic existing results about idempotents and nilpotent elements in $R(M)$ and study the nil clean ideals of the idealization $R(M)$ of $R$ and $R$-module $M$.
\begin{Lem}\label{RM}
  Let $R$ be a commutative ring and $R(M)$ be the idealization of $R$ and $R$-module $M$. Then the following hold:
  \begin{enumerate}
    \item  $(r,m) \in Idem(R(M))$ if and only if $r \in Idem(R)$ and $m =0$.
    \item  $(r,m) \in Nil(R(M))$ if and only if $r \in Nil(R)$.
  \end{enumerate}
\end{Lem}
\begin{proof}
  (1) is obvious and (2) follows from the fact that $(r,m)^n=(r^n, nr^{n-1}m)$, for any $r\in R$ and $m\in M$.
\end{proof}
\begin{Prop}\label{RM1}
Let $R$ be a commutative ring and $R(M)$, the idealization of $R$ and $R$-module $M$. Then an ideal $I$ of $R$ is weak nil clean ideal of $R$ \textit{if and only if} $I(N)$ is weak nil clean ideal of $R(M)$, for any submodule $N$ of $M$.
\end{Prop}
\begin{proof}
  $(\Rightarrow)$ Let $I$ be weak nil clean ideal of $R$. Consider an Ideal $I(N)$ of $R(M)$, for some submodule $N$ of $M$. Let $(x,m)\in I(N)$. Then either $x=n+e$ or $x=-e+n$, for some $e\in Idem(R)$ and $n\in Nil(R)$. So either $(x,m)=(e,0)+(n,m)$ or $(x,m)=-(e,0)+(n,m)$, where $(e,0)\in Idem(R(M))$ and $(n,m)\in Nil(R(M))$ by Lemma \ref{RM}.\\
  $(\Leftarrow)$ Let $I(N)$ be a weak nil clean ideal of $R(M)$ and $r\in I$. For $(r,0)\in I(N)$, either $(r,0)=(e,0)+(n,0)$ or $(r,0)=-(e,0)+(n,0)$, for some $(e,0)\in Idem(R(M))$ and $(n,0)\in Nil(R(M))$. By Lemma \ref{RM}, we conclude that either $r=e+n$ or $r=-e+n$, where $e\in Idem(R)$ and $n\in Nil(R)$.
\end{proof}
In the following proposition we study about weak nil clean element of a corner ring.
\begin{Prop}
  Let $R$ be a ring and $f\in Idem(R)$. An element $a\in fRf$ is strongly weak nil clean element of $R$ \textit{if and only if} $a\in fRf$ is strongly weak nil clean element of $fRf$.
\end{Prop}
\begin{proof}
  As $fRf$ is a left ideal of $fR$ and $fR$ is a right ideal of $R$. Hence the result follows from Theorem \ref{111}.
\end{proof}
\begin{Cor}\label{cor1}
  If $R$ is strongly weak nil clean ring and $f\in R$ is any idempotent, then the corner ring $fRf$ is strongly weak nil clean.
\end{Cor}
A Morita context denoted by $(R,S,M,N,\psi,\phi)$ consists of two rings $R$ and $S$, two bimodules $_RN_S$ and $_SM_R$ and a pair of bimodule homomorphisms (called pairings) $\psi:N\otimes _SM\rightarrow R$ and $\phi:M\otimes _RN\rightarrow S$, which satisfy the following associativity: $\psi(n\otimes m)n'=n\phi(m\otimes n')$ and $\phi (m\otimes n)m'=m\psi(n\otimes m')$, for any $m,\,m'\in M$ and $n,\,n'\in N$. These conditions ensure that the set of matrices $\left(
                                                                 \begin{array}{cc}
                                                                   r & n \\
                                                                   m & s \\
                                                                 \end{array}
                                                               \right)$, where $r\in R$, $s\in S$, $m\in M$ and $n\in N$ forms a ring denoted by $T$, called the ring of the context. For any subset $I$ of $T$, define $p_R(I)=\{a\in R\,:\,\left(
                                                                 \begin{array}{cc}
                                                                   a & x\\
                                                                   y & b \\
                                                                 \end{array}
                                                               \right)\in I \}$, $p_M(I)=\{y\in M\,:\,\left(
                                                                 \begin{array}{cc}
                                                                   a & x\\
                                                                   y & b \\
                                                                 \end{array}
                                                               \right)\in I \}$, $p_S(I)=\{b\in S\,:\,\left(
                                                                 \begin{array}{cc}
                                                                   a & x\\
                                                                   y & b \\
                                                                 \end{array}
                                                               \right)\in I \}$ and $p_N(I)=\{x\in N\,:\,\left(
                                                                 \begin{array}{cc}
                                                                   a & x\\
                                                                   y & b \\
                                                                 \end{array}
                                                               \right)\in I \}$\par
A morita context $R= \left(
                                                                 \begin{array}{cc}
                                                                   A & M \\
                                                                   N & B \\
                                                                 \end{array}
                                                               \right)$ is called morita context of zero pairing if context products $MN=0$ and $NM=0$.
\begin{Lem}
  Let $T= \left(
                                                                 \begin{array}{cc}
                                                                   A & M \\
                                                                   N & B \\
                                                                 \end{array}
                                                               \right)$ be a morita context. Then $I$ is an ideal of $T$ \textit{if and only if} $I=\left(
                                                                 \begin{array}{cc}
                                                                   A_1 & M_1 \\
                                                                   N_1 & B_1 \\
                                                                 \end{array}
                                                               \right)$,
  where $A_1$ and $B_1$ are ideals of $A$ and $B$ respectively, $M_1$ and $N_1$ are submodules of $ _AM_B$ and $ _BN_A$ respectively, with $M_1N\subseteq A_1$, $N_1M\subseteq B_1$, $A_1M\subseteq M_1$, $B_1N\subseteq N_1$, $MN_1\subseteq A_1$, $NM_1\subseteq B_1$, $MB_1\subseteq M_1$ and $NA_1\subseteq N_1$. In this case $A_1=p_{A}(I)$, $B_1=p_B(I)$, $M_1=p_M(I)$ and $N_1=p_N(I)$.
\end{Lem}
\begin{proof}
  See Lemma $2.1$ \cite{Tang}.
\end{proof}
\begin{Thm}
  Let $R= \left(
                                                                 \begin{array}{cc}
                                                                   A & M \\
                                                                   N & B \\
                                                                 \end{array}
                                                               \right)$ be a morita context and $I= \left(
                                                                 \begin{array}{cc}
                                                                   A_1 & M_1 \\
                                                                   N_1 & B_1 \\
                                                                 \end{array}
                                                               \right)$ be a strongly weak nil clean ideal of $R$. Then $A_1$ and $B_1$ are strongly weak nil clean ideals of $A$ and $B$ respectively.
\end{Thm}
\begin{proof}
  The proof follows from Corollary \ref{cor1} and Theorem \ref{111}.
\end{proof}

\begin{Thm}
  Let $R= \left(
                                                                 \begin{array}{cc}
                                                                   A & M \\
                                                                   N & B \\
                                                                 \end{array}
                                                               \right)$ be a morita context of zero pairing. If $A_1$ and $B_1$ are weak nil clean ideals of $A$ and $B$ respectively, where at least one of them is strongly nil clean ideal then $I= \left(
                                                                 \begin{array}{cc}
                                                                   A_1 & M_1 \\
                                                                   N_1 & B_1 \\
                                                                 \end{array}
                                                               \right)$ is a weak nil clean ideal of $R$.
\end{Thm}

\begin{proof}
  Let $A_1$ be strongly nil clean ideal of $A$ and $B_1$ be strongly weak nil clean ideal of $B$ respectively. Let $x= \left(
                                                                 \begin{array}{cc}
                                                                   a & m \\
                                                                   n & b \\
                                                                 \end{array}
                                                               \right)\in I$.
 Then there exist $e\in Idem(B_1)$ and $q\in Nil(B_1)$ such that either $b=e+q$ or $b=-e+q$. If $b=e+q$, then set $a=f+p$, where $f\in Idem(A_1)$ and $p\in Nil(A_1)$ and we get, $x= \left(
                                                                 \begin{array}{cc}
                                                                   f & 0 \\
                                                                   0 & e \\
                                                                 \end{array}
                                                               \right)+ \left(
                                                                 \begin{array}{cc}
                                                                   p & m \\
                                                                   n & q \\
                                                                 \end{array}
                                                               \right)$
  and  $p^k=0$ and $q^k=0$, for some $k\in \mathbb{N}$, where $\left(
                                                                 \begin{array}{cc}
                                                                   f & 0 \\
                                                                   0 & e \\
                                                                 \end{array}
                                                               \right)\in Idem(R)$.
   Now from Theorem $2.28$ \cite{Ajay1}, we conclude that $\left(
                                                                 \begin{array}{cc}
                                                                   p & m \\
                                                                   n & q \\
                                                                 \end{array}
                                                               \right)\in Nil(R)$.
   Also if $b=-e+q$ then set $a=-f+p$ and similar as above we can show that $x$ is weak nil clean element of $R$.
\end{proof}

\end{document}